\documentclass[american, letterpaper]{article}

\usepackage[margin=1in]{geometry}
\makeatletter
\def\thm@space@setup{%
\thm@preskip=1em \thm@postskip=0pt
}
\makeatother

\usepackage{amsmath, amsthm, amssymb, bbold, mathtools}
\usepackage{graphicx}
\usepackage{acronym}
\usepackage{multirow}
\usepackage{import}
\usepackage{url}

\usepackage[
backend=bibtex,
style=authoryear,
natbib=true,
url=false, 
doi=false,
eprint=false,
dashed=false,
maxbibnames=99
]{biblatex}
\usepackage{hyperref}
\usepackage[]{hyperref}
\hypersetup{
  colorlinks=false,
}

\usepackage[ruled,vlined]{algorithm2e}

\newtheorem{lemma}{Lemma}

\newtheorem{theorem}{Theorem}
\newtheorem{corollary}{Corollary}

\acrodef{qop}[QOP]{Quadratic Optimization Problem}
\acrodef{socp}[SOCP]{Second-Order Cone Problem}
\acrodef{sdp}[SDO]{Semidefinite Optimization}
\acrodef{mio}[MIO]{Mixed Integer Optimization}
\acrodef{cio}[CIO]{Convex Integer Optimization}
\acrodef{kkt}[KKT]{Karush-Kuhn-Tucker}
\acrodef{amp}[AMP]{Approximate Message-Passing}
\acrodef{tsp}[TSP]{Travelling Salesman Problem}
\acrodef{iid}[i.i.d.]{independent identically distributed}
\acrodef{snr}[SNR]{signal-to-noise ratio}

\newcommand{\norm}[1]{\left\|#1\right\|}
\newcommand{\conv}{\mathop{\operatorname{conv}}}
\newcommand{\abs}[1]{\left|#1\right|}

\newcommand{\eye}[1]{\mathbb{I}_{#1}}
\newcommand{\set}[2]{\left\{ #1\ : \ #2 \right\}}
\newcommand{\tpose}{^\top}
\newcommand{\ip}[2]{#1\tpose #2}
\newcommand{\defn}[0]{:=}

\newcommand{\mc}{\mathcal}
\renewcommand{\emph}{\textit}
\newcommand{\mb}{\mathbb}
\def\d{\mathrm{d}}

\def\Re{\mathrm{R}}

\def\one{\mb 1}
\renewcommand{\S}{\mathrm{S}}
\def\st{\mathrm{s.t.}}
\def\bigM{big-$\mc M$~}

\DeclareMathOperator{\supp}{supp}

\providecommand{\keywords}[1]{\textbf{\textit{Keywords: }} #1}

\usepackage[affil-it]{authblk}

\title{Sparse High-Dimensional Regression: Exact Scalable Algorithms and Phase Transitions}

\author[1]{Dimitris Bertsimas\thanks{\href{mailto:dbertsim@mit.edu}{dbertsim@mit.edu}}}
\author[1]{Bart \mbox{Van Parys}\thanks{\href{mailto:vanparys@mit.edu}{vanparys@mit.edu}}}

\affil[1]{Operations Research Center, Massachusetts Institute of Technology}

\date{}

\begin{document}
\maketitle

\begin{abstract}
  We present a novel  binary convex reformulation of the sparse regression problem that constitutes a new duality perspective. 
  We devise a new cutting plane method and provide evidence that it can solve to provable optimality the sparse regression problem for  sample sizes  $n$ and number of regressors $p$ in the 100,000s, that is two orders of magnitude better than the current state of the art, in seconds.  The ability to solve the  problem  for very high dimensions allows us to observe  new phase transition phenomena. Contrary to traditional complexity theory which suggests that the difficulty of a problem increases with problem size, the sparse  regression problem  
  has the property that as the number of samples  $n$ increases the problem becomes easier in that the solution recovers 100\% of the true signal, and 
  our approach solves the problem extremely fast (in fact faster than \texttt{Lasso}), while 
  for small number of samples $n$, our approach takes a larger amount of time to solve 
  the problem, but importantly the optimal solution provides a statistically more relevant regressor. We argue that our exact sparse regression approach presents a superior alternative over heuristic methods available at present. 
\end{abstract}


\acresetall

%%%%%%%%%%%%%%%%%
%% SEC. Introduction
%%%%%%%%%%%%%%%%%
\section{Introduction}
\label{sec:introduction}

Given input data $X = (x_1, \dots, x_n)\in \Re^{n\times p}$ and response data $Y = (y_1, \dots,  y_n)\in \Re^n$, the problem of   linear regression
with a \citet{tikhonov1943stability} regularization term and an explicit sparsity constraint
 is defined as
\begin{equation}
\label{eq:l0-regression:primal}
%\tag{$P_{\mathrm 0}$}
\begin{array}{rl}
	\displaystyle \min_w 	& \frac{1}{2\gamma} \norm{w}^2_2 + \frac{1}{2} \norm{Y - X w}_2^2 \\[0.5em]
	\st    	& \norm{w}_0 \leq k, \\
\end{array}
\end{equation}
where $\gamma>0$ is a given weight that controls the importance of the regularization term. 
The number of regression coefficients needed to explain the observations from the input data is limited to $k$ by the $\ell_0$-norm constraint on the regressor $w$. Tikhonov regularization helps to reduce the effect of noise in the input data. Regularization and robustness are indeed known to be intimately connected as shown for instance by \citet{bertsimas2009equivalence, xu2009robustness}. Evidently in practice, both the sparsity parameter $k$ and the Tikhonov regularization term $\gamma$ must ultimately be determined from the data. Cross validation has in practice been empirically found to be an effective method to determine both hyperparameters.

\subsection*{Background}
Problem \eqref{eq:l0-regression:primal} is  a  discrete optimization problem, which belongs to the class of $NP$-hard problems.  Motivated  by the apparent difficulty of the sparse regression formulation \eqref{eq:l0-regression:primal}, much of the literature  until recently has largely ignored the exact discrete formulation and rather focused on heuristic approaches. Historically, the first heuristics methods for sparse approximation seem to have arisen in the signal processing community (c.f.\ the work of \citet{mallat1993matching} and references therein) and typically are of an iterative thresholding type. More recently, one popular class of sparse regression heuristics solve convex surrogates to the sparse regression formulation \eqref{eq:l0-regression:primal}. There is an elegant theory for such schemes promising large improvements over the more myopic iterative thresholding methods. Indeed, a truly impressive amount of high-quality work \citep{buhlmann2011statistics, hastie2015statistical, wainwright2009sharp} has been written on characterizing when exact solutions can be recovered, albeit through making strong assumptions on the data.

One such heuristic based on a convex proxy related to our formulation and particularly worthy of mention is the \texttt{Elastic Net} developed by \citet{zou2005regularization}. One particular canonical form of the \texttt{Elastic Net} heuristic solves the proxy convex optimization problem
\begin{equation}
\label{eq:l1-regression:primal}
%\tag{$P_{\mathrm 1}$}
\begin{array}{rl}
	\displaystyle \min_w 	& \frac{1}{2\gamma} \norm{w}^2_2 + \frac{1}{2} \norm{Y - X w}_2^2 \\[0.5em]
	\st    	& \norm{w}_1 \leq \lambda,
\end{array}
\end{equation}
where the $\ell_1$-norm constraint shrinks the regressor coefficients towards zero thus encouraging sparse regressors for $\lambda$ tending to zero. When disregarding the Tikhonov regularization term, the popular \texttt{Lasso} heuristic introduced by \citet{tibshirani} is recovered. 
An important factor in favor of heuristics such as \texttt{Lasso} and \texttt{Elastic Net} are their computational feasibility and scalability. Indeed, problem \eqref{eq:l1-regression:primal} can be solved efficiently and mature software implementations such as \texttt{GLMNet} by \citet{friedman2013glmnet} are available.

Despite all of the aforementioned positive properties, proxy based methods such as \texttt{Lasso} and \texttt{Elastic Net}  do have several innate shortcomings. These shortcomings are well known in the statistical community too. First and foremost, as argued in \citep{bertsimas2014statistics} they do not recover very well the sparsity pattern. Furthermore, the \texttt{Lasso} leads to biased regression regressors, since the $\ell_1$-norm penalizes both large and small coefficients uniformly. In sharp contrast, the $\ell_0$-norm sparsifies the regressor without conflating the effort with unwanted shrinking. 

For a few decades the exercise of trying to solve the sparse regression problem \eqref{eq:l0-regression:primal} at a practical scale was branded hopeless. \citet{bixby2012brief} noted however that in the last twenty-five years the computational power of \ac{mio} solvers has increased at an astonishing rate. Riding on the explosive improvement of \ac{mio} formulations, \citet{bertsimas2014statistics} achieved to solve the sparse regression problem \eqref{eq:l0-regression:primal} for  problem instances of dimensions $n$, $p$ in the 1000s. Using a \bigM formulation of the cardinality constraint, the sparse regression problem \eqref{eq:l0-regression:primal} can indeed be transformed into the \ac{mio} problem
\begin{equation}
\label{eq:bigm:primal}
	\begin{array}{rl}
	\min 	& \frac{1}{2\gamma} \norm{w}_2^2 + \frac{1}{2} \norm{Y - X w}_2^2 \\[0.5em]
	\st    	& w \in \Re^p, ~s \in \S^p_k \\[0.5em]
		& -\mc M s_j \leq w_j \leq \mc M s_j,  \hspace{0.60cm} \forall j \in [p].
	\end{array} 
\end{equation}
With the help of the binary set $\S^p_k \defn \set{s\in\{0, 1\}^p}{\one\tpose s \leq k}$, the constraint in \eqref{eq:bigm:primal} ensures that the regression coefficient $w_j$ is nonzero only if the selection variable $s_j=1$ for a sufficiently large constant $\mc M$. The constant $\mc M$ must be estimated from data as outlined in \cite{bertsimas2014statistics} to ensure the equivalence between the sparse regression problem \eqref{eq:l0-regression:primal} and its \ac{mio} formulation \eqref{eq:bigm:primal}. This \ac{mio} approach is significantly more scalable than the leaps and bounds algorithm outlined in \cite{furnival2000regressions}, largely because of the advances in computer hardware, the improvements in \ac{mio} solvers, and the specific warm-start techniques developed by \cite{bertsimas2014statistics}. Even so, many problems of practical size are still far beyond the scale made  tractable through this approach.

\subsection*{A scalable perspective}

Although a direct \bigM formulation of the sparse regression problem results in a well posed \ac{mio} problem, the constant $\mc M$ needs to be chosen with  care as not to impede its numerical solution. The choice of this data dependent constant $\mc M$ indeed affects the strength of the \ac{mio} formulation \eqref{eq:bigm:primal} and is critical for obtaining solutions quickly in practice. Furthermore, as the regression dimension $p$ grows, explicitly constructing the \ac{mio} problem \eqref{eq:bigm:primal}, let alone solving it, becomes burdensome. In order to develop an exact scalable method to the sparse regression problem \eqref{eq:l0-regression:primal} capable of solving problem instances of sample size $n$ and regressor dimension in the 100,000s, a different perspective on sparse regression is needed.

The \bigM formulation \eqref{eq:bigm:primal} of the sparse linear regression problem \eqref{eq:l0-regression:primal} takes on a primal perspective to regression. Like most exact as well as heuristic sparse regression formulations, the \bigM formulation \eqref{eq:bigm:primal} indeed tries to solve for the optimal regression coefficients $w_0^\star$ in \eqref{eq:l0-regression:primal} directly. However, it is well known in the kernel learning community that often far deeper results can be obtained if a dual perspective is taken. We show that this dual perspective can be translated to a sparse regression context as well and offers a novel road to approach exact sparse regression. Taking this new perspective, sparse regression problem \eqref{eq:l0-regression:primal} can be reduced to a pure integer convex optimization problem avoiding the  construction of any auxiliary constants.

Crucially, a tailored cutting plane algorithm for the resulting \ac{cio} problem renders solving the sparse regression problem \eqref{eq:l0-regression:primal} to optimality tractable for problem instances with number of samples and regressors in the 100,000s. That is two orders of magnitude better than the current state of art and impeaches the primary selling point of heuristic approaches such as \texttt{Elastic Net} or \texttt{Lasso}. 
As we will discuss subsequently, our cutting plane algorithm is often comparable or indeed even faster than the aforementioned convex proxy heuristic approaches.

\subsection*{Phase Transitions}

Let the data come from $Y = X w_\mathrm{true} + E$ where $E$ is zero mean noise uncorrelated with the signal $X w_\mathrm{true}$, then we define the accuracy and false alarm rate of a certain  solution $w^\star$ in recovering the correct support as: 
$$A \% \defn 100\times\frac{\abs{ \supp(w_\mathrm{true}) \cap \supp(w^\star) }}{k}$$ and $$F\% \defn 100\times\frac{\abs{ \supp(w^\star) \setminus \supp(w_\mathrm{true}) }}{\abs{\supp(w^\star)}}.$$ Perfect support recovery occurs only then when $w^\star$ tells the whole truth ($A\%=100$) and nothing but the truth ($F\%=0$).

The ability to recover the support of the ground truth $w_{\mathrm{true}}$ of the \texttt{Lasso} heuristic \eqref{eq:l1-regression:primal} for some value of $\lambda$ was shown by \citet{donoho2009observed} to experience a phase transition.
The phase transition described by \citet{donoho2009observed} concerns the ability of the \texttt{Lasso} solution $w_1^\star$ to coincide in support with the ground truth $w_{true}$. This accuracy  phase transition for the \texttt{Lasso} has been extensively studied in \citep{buhlmann2011statistics, hastie2015statistical, wainwright2009sharp} and is considered well understood by now. That being said, the assumptions made on the data needed for a theoretical justification of such phase transition are quite stringent and often of limited practical nature. For instance, \citet{wainwright2009sharp} showed that for observations $Y$ and independent Gaussian input data $X$ a phase transition occurs at the phase transition curve
\begin{equation}
  \label{eq:wainwright}
  n_1 = (2k+\sigma^2) \log (p-k),
\end{equation}
where $\sigma$ presents the noise level corrupting the observations.      
In the regime $n > n_1$ exact recovery of the support occurs with high-probability, while on the other side of the transition curve the probability for successful recovery drops to zero. Nonetheless, this  phase transition from accurate discovery to statistical meaninglessness  has been widely observed empirically \citep{donoho2009observed, donoho2006breakdown} even under conditions in which these assumptions are severely violated.

For exact sparse regression \eqref{eq:l0-regression:primal} a similar phase transition has been observed by \citet{zheng2015does} and \citet{wang2011performance}, although this transition is far less studied from a theoretical perspective than the similar transition for its heuristic counterpart. It is however known that the accuracy phase transition for exact sparse regression must occur even sooner than that of any heuristic approach. That is, exact sparse regression \eqref{eq:l0-regression:primal} yields statistically more meaningful optima than for instance the convex \texttt{Lasso} heuristic \eqref{eq:l1-regression:primal} does. Recently \citet{gamarnik2017high}, motivated by the results of the present paper, showed that when the regression coefficients are binary, a phase transition occurs at
\begin{equation}
  \label{eq:gamarnik}
  n_0 = 2 k \log{p} / \log\left(\frac{2k}{\sigma^2} + 1\right).
\end{equation}
Empirical verification of this phase transition was historically hindered due to the lack of exact scalable algorithms. Our novel cutting plane algorithm lifts this hurdle and opens the way to show the benefits of exact sparse regression empirically.

More importantly, we present strong empirical evidence that a computational phase transition occurs as well. Specifically, there is a phase transition concerning our ability to solve the sparse regression problem \eqref{eq:l0-regression:primal} efficiently. In other words, there is a phase transition in our ability to recover the true coefficients of the sparse regression problem and most surprisingly in our ability to find them fast. This complexity phase transition does not seem to be reported before and sheds a new light on the complexity of sparse  linear regression. 
Contrary to traditional complexity theory which suggests that the difficulty of a problem increases with size, the sparse  regression problem \eqref{eq:l0-regression:primal} has the property that for a small number of samples $n < n_t$, our approach takes a large amount of time to solve the problem. However, for a large number of samples $n > n_t$, our approach solves 
the problem extremely fast and perfectly recovers the support of the true regressor $w_{\mathrm{true}}$ fully. The complexity phase transition occurs between the theoretically minimum amount of samples $n_0 < n_t$ needed by exact sparse regression, there remains some hardness to the problem after all, but occurs crucially before $n_t<n_1$ the \texttt{Lasso} heuristic provides statically meaningful regressors. 

Lastly, recall that the accuracy phase transition \eqref{eq:wainwright} for \texttt{Lasso} and its counterpart \eqref{eq:gamarnik} for exact sparse regression are applicable only then when the true sparsity $k$ is known. Evidently in practice, the sparsity parameter $k$ must ultimately be determined from the data. Most commonly this is done using cross validation. Incorrect determination of this parameter most often leads to elevated false alarm rates. Crucially, we show that in this regard only exact sparse regression experiences a phase transitions in its ability to select only the relevant features. Lasso always seems to favor adding irrelevant features in an attempt to improve its prediction performance. We will show that exact regression is significantly better than \texttt{Lasso} in discovering all true relevant features ($A\%= 100$), while truly outperforming its ability to reject the obfuscating ones ($F\%= 0$).

\subsection*{Contributions and structure}
\begin{enumerate}
\item In Section \ref{sec:sparse_regression}, we propose a novel  binary convex reformulation of the sparse regression problem \eqref{eq:l0-regression:primal} that represents a new dual perspective to the problem. 
The reformulation does not use the \bigM constant present in the primal  formulation \eqref{eq:bigm:primal}. In Section \ref{sec:cutting_plane}, we devise a novel cutting plane method and provide evidence that it can solve the sparse regression problem for  sizes of $n$ and $p$ in the 100,000s. That is two orders of magnitude than what was achieved in \citep{bertsimas2014statistics}. The empirical computational results in this paper do away with the long held belief that exact sparse regression for practical problem sizes is a lost cause.
\item The ability to solve the sparse regression problem \eqref{eq:l0-regression:primal}  for very high dimensional problems allows us to observe properties of the problem 
that demonstrate new phase transition phenomena.
Specifically,  we demonstrate experimentally  in Section \ref{sec:empirical_results} that  there is a threshold $n_t$ such that if $n \geq n_t$, then $w^\star_0$ recovers the true support ($A\%=100$ for $F\%=0$) and the time to solve problem \eqref{eq:l0-regression:primal}  is 
 seconds 
(for $n$ and $p$ in 100,000s) and it only grows only linear in $n$. Remarkably, these times are less than the time to solve \texttt{Lasso} for similar sizes. 
Moreover, if $n < n_t$, then the time to solve problem \eqref{eq:l0-regression:primal} grows proportional to $\binom{p}{k}$. In other words, there is a phase transition in our ability to recover the true coefficients of the sparse regression problem and most surprisingly in our ability to solve it. Contrary to traditional complexity theory that suggests that the difficulty of a problem increases with dimension, the sparse  regression problem \eqref{eq:l0-regression:primal} 
has the property that for small number of samples $n$, our approach takes a  large amount of time to solve the problem, but most importantly the optimal solution does not recover the true signal. However, for a large number of samples $n$, our approach solves 
the problem extremely fast and recovers $A\%=100$ of the support of the true regressor $w_{\mathrm{true}}$.  Significantly, the threshold $n_t$ for the phase transition for full recovery 
of exact sparse regression is significantly smaller than the corresponding threshold $n_1$ for  \texttt{Lasso}. Whereas \texttt{Lasso} tends to furthermore include many irrelevant features as well, exact sparse regression furthermore achieves this full recovery at almost $F\%=0$ false alarm rate.
\item We are able to generalize in Section \ref{sec:nonlinear} our approach to sparse kernel regression. We believe that this nonlinear approach can become a fierce and more disciplined competitor compared to  ``black box'' approaches such as neural networks.
\end{enumerate}

\subsection*{Notation}

Denote with $[n]$ the set of integers ranging from one to $n$. The set $\S^p_k$ denotes the   set $$\S_k^p\defn\set{s \in \{0,1\}^p}{\one\tpose s\leq k},$$  which contains all binary vectors $s$ selecting $k$ components from $p$ possibilities. Assume that $(y_1, \dots, y_p)$ is a collection of elements and suppose that $s$ is an element of $\S^p_k$, then $y_s$ denotes the sub-collection of $y_j$ where $s_j = 1$. We use $\norm{x}_0$ to denote the number of elements of a vector $x$ in $\Re^p$ which are nonzero. Similarly, we use $\supp(x) = \set{s \in \{0, 1\}^p}{s_i = 1 \iff x_i \neq 0} $ to denote those indices of a vector $x$ which are nonzero. Finally, we denote by $\S_+^n$ ($\S_{++}^n$) the cone of $n\times n$ positive semidefinite  (definite) matrices.

%%%%%%%%%%%%%%%%%
%% SEC. A convex binary reformulation of sparse linear regression
%%%%%%%%%%%%%%%%%
\section{A convex binary reformulation of sparse linear regression}
\label{sec:sparse_regression}

Sparse regression taken at face value is recognized as a mixed continuous and discrete optimization problem. Indeed, the sparse regressor $w$ as an optimization variable in \eqref{eq:l0-regression:primal} takes values in a continuous subset of $\Re^p$. The $\ell_0$-norm sparsity constraint, however, adds a discrete element to the problem. The support $s$ of the sparse regressor $w$ is discrete as it takes values in the binary set $\S^p_k=\set{s \in \{0,1\}^p}{\one\tpose s\leq k}$. It should not come as a surprise then that the reformulation \eqref{eq:bigm:primal} developed by \citet{bertsimas2014statistics} formulates the sparse regression problem as a \ac{mio} problem.  
 
For the reasons outlined in the introduction of this paper, we take a different approach to the sparse regression problem \eqref{eq:l0-regression:primal} entirely. To that end we first briefly return to the ordinary regression problem for which any sparsity considerations are ignored and in which a linear relationship between input data $X$ and observations $Y$ is determined through solving the least squares regression problem
\begin{equation}
\label{eq:regression:primal}
\begin{array}{rl}
	c \defn \min 	& \frac{1}{2\gamma} \norm{w}_2^2 + \frac{1}{2} \norm{Y - X w}^2_2 \\[0.5em]
	\st    	& w \in \Re^p.
\end{array}
\end{equation}
We will refer to the previously defined quantity $c$ as the regression loss. The quantity $c$ does indeed agree with the regularized empirical regression loss for the optimal linear regressor corresponding to the input data $X$ and response $Y$.
We point out now that the regression loss function $c$ is convex as a function of the outer product $X X\tpose$ and furthermore show that it admits an explicit characterization as a semidefinite representable function. 

\begin{lemma}[The regression loss function $c$]
\label{lemm:convexity}
The regression loss $c$ admits the following explicit characterizations
\begin{align}
		c 	&= \frac12 Y\tpose \left( \eye{n} -  X \left(\eye{p}/\gamma + X\tpose X\right)^{-1} X\tpose \right) Y, \label{eq:char1} \\[0.5em]
			&= \frac12 Y\tpose \left(\eye{n} + \gamma X X\tpose \right)^{-1} Y. \label{eq:char2} \\
\intertext{Futhermore, the regression loss $c$ as a function of the kernel matrix $X X\tpose$ is conic representable using the formulation}
		c(X X\tpose)	&=\min \set{\eta\in \Re_+}{\begin{pmatrix} 2 \eta & Y\tpose \\ Y & \eye{n} + \gamma X X\tpose \end{pmatrix} \in \S^{n+1}_+}. 
		\label{eq:char3} 
\end{align}
\end{lemma}
\begin{proof}
As the minimization problem \eqref{eq:regression:primal} over $w$ in $\Re^p$ is an unconstrained \ac{qop}, the optimal value $w^\star$ satisfies the linear relationship $(\eye{p}/\gamma + X\tpose X)w^\star = X\tpose Y.$ Substituting the expression for the optimal linear regressor $w^\star$ back into optimization problem, we arrive at 
\[
	c = \frac12 Y\tpose Y - \frac12 Y\tpose X \left(\eye{p}/\gamma + X\tpose X\right)^{-1} X\tpose Y
\]
establishing the first explicit characterization \eqref{eq:char1} of the regression function $c$.  The second characterization \eqref{eq:char2} can be derived from the first with the help of the matrix inversion lemma found in \cite{hager1989updating} stating the identity 
\[
	\left(\eye{n} + \gamma X X\tpose\right)^{-1} = \eye{n} - X \left( \eye{p}/\gamma + X\tpose X\right)^{-1} X\tpose.
\]
The Schur complement condition discussed at length in \cite{zhang2006schur} guarantees that as $\eye{n} + \gamma X X\tpose$ is strictly positive definite, we have the equivalence
\[
	2 \eta \geq Y\tpose \left(\eye{n} + \gamma X X\tpose\right)^{-1} Y \iff \begin{pmatrix} 2 \eta & Y\tpose \\ Y & \eye{n} + \gamma X X\tpose \end{pmatrix} \in \S^{n+1}_+.
\]%
Representation \eqref{eq:char3} is thus an immediate consequence of expression \eqref{eq:char2} as well.
\end{proof}

We next   establish that the sparse regression problem \eqref{eq:l0-regression:primal} can in fact be represented as a pure binary  optimization problem. The following result provides a novel perspective on the sparse regression problem \eqref{eq:l0-regression:primal} and is of central importance in the paper.

\begin{theorem}[Sparse linear regression]
\label{thm:cio}
The sparse regression problem \eqref{eq:l0-regression:primal} can be reformulated as the nonlinear optimization problem 
\begin{equation}
\label{eq:opt:miop:kernel}
\begin{array}{rl}
	\min 	& \displaystyle \frac12 Y\tpose \left(\eye{n} + \gamma \textstyle\sum_{j\in[p]} s_j K_j \right)^{-1} Y \\[0.6em]
	\st    	& s\in \S^p_k,
\end{array}
\end{equation}
where the micro kernel matrices $K_j$ in $\S^n_+$ are defined as the dyadic products
\begin{equation}
\label{eq:kernel}
	\textstyle K_j \defn X_j X_j \tpose.
\end{equation}
\end{theorem}
\begin{proof}
We start the proof by separating the optimization variable $w$ in the sparse regression problem \eqref{eq:l0-regression:primal} into its support $s \defn \supp{w}$ and the corresponding non-zero entries $w_s$. Evidently, we can now write the sparse regression problem \eqref{eq:l0-regression:primal} as the bilevel minimization problem 
\begin{equation}
\label{eq:reformulation}
\min_{s\in\S^p_k}\left[\min_{w_s \in \Re^k} ~ \frac{1}{2\gamma} \norm{w_s}^2_2 + \frac{1}{2} \norm{Y - X_s w_s}_2^2 \right].
\end{equation}
It now remains to be shown that the inner minimum can be found explicitly as the objective function of the optimization problem \eqref{eq:opt:miop:kernel}. Using Lemma \ref{lemm:convexity}, the minimization problem can be reduced to the binary minimization problem $\min_s \set{c(X_s X_s\tpose)}{s\in\S^p_k}$. We finally remark that the outer product can be decomposed as the sum $$X_s X_s\tpose = \textstyle\sum_{j\in[p]} s_j X_j X_j\tpose,$$ thereby completing the proof.
\end{proof}

An alternative to  the sparse regression problem \eqref{eq:l0-regression:primal}   is to consider  
the  penalized form of the  sparse regression problem: 
\begin{equation}
  \label{probl:penalized:l0}
  \begin{array}{rl}
    \displaystyle\min_{w\in \Re^p} & \frac12 \norm{Y- Xw}_2^2 + \frac{1}{2\gamma} \norm{w}_2^2 + \lambda \norm{w}_0,
  \end{array}
\end{equation}
in which the $\ell_0$-norm constraint is migrated to the objective function.  Analogously to  Theorem \ref{thm:cio} we can show that 
problem \eqref{probl:penalized:l0} can be reformulated as the nonlinear optimization problem 
\begin{equation*}
\begin{array}{rl}
	\min 	& \displaystyle \frac12 Y\tpose \left(\eye{n} + \gamma \textstyle\sum_{j\in[p]} s_j K_j \right)^{-1} Y  + \lambda \cdot \one\tpose s  \\[0.5em]
        \st 	& s \in \{0, 1\}^p.\\[0.5em]
\end{array}
\end{equation*}
While we do not need to  pre-specify $k$ in problem \eqref{probl:penalized:l0}, we need to specify the penalty $\lambda$ instead. 

The optimization problem \eqref{eq:opt:miop:kernel} is a pure binary formulation of the sparse regression problem directly over the support $s$ instead of the regressor $w$ itself. As the objective function in \eqref{eq:opt:miop:kernel} is convex in the vector $s$,  problem \eqref{eq:opt:miop:kernel} casts the sparse regression problem as a \ac{cio} problem. Nevertheless, we will never explicitly construct the \ac{cio} formulation as such and rather develop in Section \ref{sec:cutting_plane} an efficient cutting plane algorithm. 
We finally discuss here how the sparse regression formulation in Theorem \ref{thm:cio} is related to kernel regression and admits an interesting dual relaxation.

\subsection{The kernel connection}

In ordinary linear regression a linear relationship between input data $X$ and observations $Y$ is determined through solving the least squares regression problem \eqref{eq:regression:primal}. The previous optimization problem is known as Ridge regression as well and balances the least-squares prediction error with a Tikhonov regularization term. One can solve the Ridge regression problem in the primal space -- the space of parameters $w$ -- directly. Ridge regression is indeed easily recognized to be a convex \ac{qop}. Ordinary linear regression problems can thus be formulated as \ac{qop}s
 of size linear in the number of regression coefficients $p$. 

Correspondingly, the \bigM formulation \eqref{eq:bigm:primal} can be regarded as a primal perspective on the sparse regression problem \eqref{eq:l0-regression:primal}. Formulation \eqref{eq:bigm:primal} indeed attempts to solve the sparse regression problem in the primal space of parameters $w$ directly. 

However, it is well known in the kernel learning community that far deeper results can be obtained if one approaches regression problems from its convex dual perspective due to \citet{vapnik1998support}. Indeed, in most of the linear regression literature the dual perspective is often preferred over its primal counterpart. We state here the central result  in this context to make the exposition self contained.

\begin{theorem}[{\citet{vapnik1998support}}]
\label{thm:vapnik}
The primal regression problem \eqref{eq:regression:primal} can equivalently be formulated as the unconstrained maximization problem 
\begin{equation}
\label{eq:regression:dual}
%\tag{$D$}x
\begin{array}{rl}
	c = \max	& -\frac{\gamma}{2} \alpha\tpose K \alpha  - \frac{1}{2} \alpha\tpose \alpha + Y\tpose \alpha  \\[0.5em]
	\st    	& \alpha \in \Re^n, \\
\end{array}
\end{equation}
where the kernel matrix $K = X X\tpose$ in $\S^n_+$ is a positive semidefinite matrix.
\end{theorem}
The dual optimization problem \eqref{eq:regression:dual} is a convex \ac{qop}
 as well and, surprisingly, scales only with the number of samples $n$ and is insensitive to the input dimension $p$. This last surprising observation is what gives the dual perspective its historical dominance over its primal counterpart in the context of kernelized regression discussed in \citep{scholkopf2002learning}. When working with high dimensional data for which the number of inputs $p$ is vastly bigger than the number of samples $n$, the dual optimization problem \eqref{eq:regression:dual} is smaller and often easier to solve.

For any $i$ and $j$, the kernel matrix entry $K(i, j)$ corresponds to the inner product between input samples $x_i$ and $x_j$ in $\Re^p$. The matrix $K$ is usually referred to as the kernel matrix or Gram matrix and is always positive definite and symmetric. Since the kernel specifies the inner products between all pairs of sample points in $X$, it completely determines the relative positions of those points in the embedding space.

Our \ac{cio} formulation \eqref{eq:opt:miop:kernel} of the sparse optimization problem \eqref{eq:l0-regression:primal} can be seen to take a dual perspective on the sparse regression problem \eqref{eq:l0-regression:primal}. That is, our novel optimization formulation \eqref{eq:opt:miop:kernel} is recognized as a subset selection problem in the space of kernels instead of regressors. It can indeed be remarked that when the sparsity constraint is omitted the kernel matrix reduces to the standard kernel matrix $$K = \textstyle\sum_{j\in[p]} X_j X_j\tpose = X X\tpose.$$

\subsection{A second-order cone relaxation}

Many heuristics approach the sparse regression problem \eqref{eq:l0-regression:primal} through a continuous relaxation. Indeed, a continuous relaxation of the \bigM formulation \eqref{eq:bigm:primal} of the sparse regression problem is immediately recognized as the convex  \ac{qop}
\begin{equation}
\label{eq:bigm:relaxation}
	\begin{array}{rl}
	\displaystyle \min_w 	& \frac{1}{2\gamma} \norm{w}_2^2 + \frac{1}{2} \norm{Y - X w}_2^2 \\[0.5em]
	\st    	& \norm{w}_\infty \leq \mc M, ~\norm{w}_1 \leq \mc{M} k
	\end{array} 
\end{equation}
which \citet{bertsimas2014statistics} recognized as a slightly stronger relaxation than the \texttt{Elastic Net} \eqref{eq:l1-regression:primal}. It thus makes sense to look at the continuous relaxation of the sparse kernel optimization problem \eqref{eq:opt:miop:kernel} as well. Note that both the \bigM \eqref{eq:bigm:relaxation} and \texttt{Elastic Net} \eqref{eq:l1-regression:primal} relaxation provide lower bounds to the exact sparse regression problem \eqref{eq:l0-regression:primal} in terms of a \ac{qop}. However, neither of the these relaxations is  very tight. In Theorem \ref{thm:cio:relaxation} we will indicate that a more intuitive and comprehensive lower bound based on our \ac{cio} formulation \eqref{eq:opt:miop:kernel} can be stated as a \ac{socp}.

A naive attempt to state a continuous relaxation of the \ac{cio} formulation \eqref{eq:opt:miop:kernel} in which we would replace the binary set $\S_k^p$ with its convex hull would result in a large but convex \ac{sdp} problem. Indeed, the convex hull of the  set $\S^p_k$ is the convex polytope $\{s\in[0,1]^p:\one\tpose s \leq k\}$. It is, however,
 folklore that large \ac{sdp}s are notoriously difficult to solve in practice. 
 For this reason, we reformulate here the continuous relaxation of \eqref{eq:opt:miop:kernel} as  a small \ac{socp} for which very efficient solvers do exist. This continuous relaxation provides furthermore some additional insight  towards the binary  formulation of the sparse regression problem \eqref{eq:l0-regression:primal}.

Using Theorem \ref{thm:vapnik}, we can equate the continuous relaxation of problem \eqref{eq:opt:miop:kernel} to the following saddle point problem
\begin{equation}
\label{eq:saddle_point}
	\min_{ s \in \conv( \S^p_k)} \, \max_{\alpha \in \Re^n}  \, -\frac{\gamma}{2} \textstyle \sum_{j\in[p]} s_j \cdot \left[ \alpha\tpose  K_j \alpha \right]  - \frac{1}{2} \alpha\tpose \alpha + Y\tpose \alpha .
\end{equation}
Note that the saddle point function is linear in $s$ for any fixed $\alpha$ and concave continuous in $\alpha$ for any fixed $s$ in the compact set $\conv(\S^p_k)$. 
It then follows (see \citet{sion1958general}) that we can exchange the minimum and maximum operators.
 By doing so, the continuous relaxation of our \ac{cio} problem satisfies
\begin{equation}
\label{eq:continuous_relaxation}
\begin{split}
	 \min_{s\in\conv(\S^p_k)} \, c&(\textstyle\sum_{j\in[p]} s_j K_j) = \\
	& \max_{\alpha \in \Re^n} - \frac{1}{2} \alpha\tpose \alpha + Y\tpose \alpha - \frac{\gamma}{2} \max_{s \in \conv(\S^p_k)}  \,\textstyle \sum_{j\in[p]} s_j \cdot \alpha\tpose K_j \alpha.
\end{split}
\end{equation}
The inner  maximization problem admits an explicit representation as the sum of the $k$-largest components in the vector with components $\alpha\tpose K_j \,\alpha$ ranging over $j$ in $[p]$. It is thus worth noting that this continuous relaxation has a discrete element to it. The continuous relaxation of the \ac{mio} problem \eqref{eq:opt:miop:kernel} can furthermore be written down as a tractable \ac{socp}.

\begin{theorem}
\label{thm:cio:relaxation}
The continuous relaxation of the sparse kernel regression problem \eqref{eq:opt:miop:kernel} can be reduced to the following \ac{socp} 
\begin{equation}
\label{eq:opt:miop:dual}
\begin{array}{rl}
	\displaystyle\min_{s\in \conv (\S^p_k)}\, c(\textstyle\sum_{j\in[p]} s_j K_j) = \max	& \displaystyle-\frac{1}{2} \ip{\alpha}{\alpha} + \ip{Y}{\alpha} - \ip{\one}{u} - k t \\[0.6em]
	\st    	& \alpha \in \Re^n, ~ t\in\Re, ~ u \in \Re_+^p, \\[0.5em]
		& \displaystyle \frac2\gamma u_j \geq \alpha\tpose K_j \alpha - \frac{2}{\gamma} t, \quad \forall j\in[p].
\end{array}
\end{equation}
\end{theorem}
\begin{proof}
The continuous relaxation of the optimization problem \eqref{eq:opt:miop:kernel} was already identified as the optimization problem \eqref{eq:continuous_relaxation}. We momentarily focus on the inner maximization problem in \eqref{eq:continuous_relaxation} and show it admits a closed form expression. As the only constraint on the (continuous) selection vector $s$ is a knapsack constraint, the inner maximum is nothing but the sum of the $k$-largest terms in the objective. Hence, we have
\[
	\max_{s \in \conv(\S^p_k)} \, \textstyle \sum_{j\in[p]} s_j \cdot \alpha\tpose K_j \alpha = \max_{[k]}([\alpha\tpose K_1 \alpha, \dots, \alpha\tpose K_p \alpha]),
\]
where $\max_{[k]}$ is defined as the convex function mapping its argument to the sum of its $k$-largest components. Using standard linear optimization  duality we have
\[
\begin{array}{rlcrl}
	\max_{[k]}(x) = \max	& x\tpose s 		& =  	& \min	& k t + \one\tpose u \\[0.5em]
	\st     			& s \in \Re^p_+		&	& \st		& t \in \Re, ~ u \in \Re^p_+ \\[0.5em]
					& s \leq \one, ~\one\tpose s=k		& 	& 		& u_j \geq x_j -t, \quad \forall j \in [p].
\end{array}
\]
where $t$ and $u$ are the dual variables corresponding to the constraints in the maximization characterization of the function $\max_{[k]}$. Making use of the dual characterization of $\max_{[k]}$ in expression \eqref{eq:continuous_relaxation} gives us the desired  result.
\end{proof}

The continuous relaxation \eqref{eq:opt:miop:dual} of the sparse regression problem \eqref{eq:l0-regression:primal} discussed in this section is thus recognized as selecting the $k$-largest terms $\alpha\tpose K_j \alpha$ to construct the optimal dual lower bound. We shall find that the dual offers an excellent warm start when attempting to solve the sparse linear regression problem exactly.

%%%%%%%%%%%%%%%%%
%% SEC. A cutting plane algorithm
%%%%%%%%%%%%%%%%%
\section{A cutting plane algorithm}
\label{sec:cutting_plane}

We have formulated  the sparse regression problem \eqref{eq:l0-regression:primal} as a pure  binary  convex optimization problem in Theorem \ref{thm:cio}. Unfortunately, no commercial solvers are available which are targeted to solve \ac{cio} problems of the type \eqref{eq:opt:miop:kernel}. In this section, we discuss a tailored solver largely based on the algorithm described by \cite{duran1986outer}. The algorithm is a cutting plane approach which iteratively solves increasingly better \ac{mio} approximations to the \ac{cio} formulation \eqref{eq:opt:miop:kernel}. Furthermore, the cutting plane algorithm avoids constructing the \ac{cio} formulation \eqref{eq:opt:miop:kernel} explicitly which can prove burdensome when working with high-dimensional data. We provide numerical evidence in  Section \ref{sec:empirical_results} that the algorithm described here is  indeed  extremely efficient.

%%%%%%%%%%%%%%%%%
%% SSEC. Outer approximation algorithm
%%%%%%%%%%%%%%%%%
\subsection{Outer approximation algorithm}
\label{sec:outer-appr-algor}

In order to solve the \ac{cio} problem \eqref{eq:opt:miop:kernel}, we follow the outer approximation approach introduced by \citet{duran1986outer}. The algorithm described by \citet{duran1986outer} proceeds to find a solution to the \ac{cio} problem \eqref{eq:opt:miop:kernel} by constructing a sequence of \ac{mio} approximations based on cutting planes. In pseudocode, it can be seen to construct a piece-wise affine lower bound to the convex regression loss function $c$ defined in equation \eqref{eq:char3}.

\begin{algorithm}
\SetKwInOut{Input}{input}
\SetKwInOut{Output}{output}
\caption{The outer approximation process}
\label{alg:outer_approximation}
 \Input{$Y \in \Re^{n}$, $X \in \Re^{n\times p}$ and $k \in [1, p]$}
 \Output{$s^\star \in \S^p_k$ and $w^\star \in \Re^p$}
 $s_1 \gets$ warm start \\ $\eta_1 \gets 0$ \\ $t \gets 1$ \\
 \While{$\eta_t < c(s_t)$}{
   $ s_{t+1}, ~\eta_{t+1} \gets \arg \min_{s, \, \eta} \, \{ \, \eta \in \Re_+ ~\mathrm{s.t.} ~s \in \S^p_k, ~~\eta \geq c(s_i) + \nabla c(s_i) (s-s_i), ~ \forall i \in [t] \}$ \\
   $t \gets t+1$
  }
 $s^\star \gets s_{t}$ \\
 $w^\star \gets 0$, \quad $w^\star_{s^\star} \gets \left(\eye{p}/\gamma + X_{s^\star}\tpose X_{s^\star} \right)^{-1} X_{s^\star}\tpose Y$
\end{algorithm}

At each iteration, the cutting plane added $\eta \geq c(s_t) + \nabla c (s_t) (s-s_t)$ cuts off the current binary solution $s_t$ unless it happened to be optimal in \eqref{eq:opt:miop:kernel}.  As the algorithm progresses, the outer approximation function $c_t$ thus constructed $$c_t(s) \defn \max_{i\in[t]} \, c(s_t) + \nabla c(s_t) (s-s_t)$$ becomes an increasingly better approximation to the regression loss function $c$ of interest. Unless the current binary solution $s_t$ is optimal, a new cutting plane will refine the feasible region of the problem by cutting off the current feasible binary  solution.

\begin{theorem}[Cutting Plane Method]
  The procedure described in Algorithm \ref{alg:outer_approximation} terminates after a finite number of cutting planes and returns the exact sparse regression solution $w_0^\star$ of \eqref{eq:l0-regression:primal}.
\end{theorem}

Despite the previous encouraging corollary of a result found in \citep{fletcher1994solving}, it nevertheless remains the case that from a theoretical point of view exponentially many cutting planes need to be computed in the worst-case, potentially rendering our approach impractical. Furthermore, at each iteration a \ac{mio} problem needs to be solved. This can be done by constructing a branch-and-bound tree, c.f. \citet{lawler1966branch}, which itself requires a potential exponential number of leaves to be explored. This complexity behavior is however to be expected as exact sparse regression is known to be an $NP$-hard problem. Surprisingly, the empirical timing results presented in Section \ref{sec:empirical_results} suggests that the situation is much more interesting than what complexity theory might suggest. In what remains of this section, we briefly discuss three techniques to carry out the outer approximation algorithm more efficiently than a naive implementation would.

In general, outer approximation methods are known as ``multi-tree'' methods because every time a cutting plane is added, a slightly different \ac{mio} problem is to be solved anew by constructing a branch-and-bound tree. Consecutive \ac{mio}s in Algorithm \ref{alg:outer_approximation} differ only in one additional cutting plane. Over the course of our iterative cutting plane algorithm, a naive implementation would require that multiple branch and bound trees are built in order to solve the successive \ac{mio} problems. We implement a ``single tree'' way of solving the iteration algorithm \ref{alg:outer_approximation} by using dynamic constraint generation, known in the optimization literature as either a lazy constraint or column generation method. Lazy constraint formulations described in \citep{barnhart1998branch} dynamically add cutting planes to the model whenever a binary feasible solution is found. This saves the rework of rebuilding a new branch-and-bound tree every time a new binary  solution is found in Algorithm \ref{alg:outer_approximation}. Lazy constraint callbacks are a relatively new type of callback. To date, the only commercial solvers which provide lazy constraint callback functionality are \texttt{CPLEX}, \texttt{Gurobi} and \texttt{GLPK}.

In what follows, we discuss two additional tailored adjustments to the general outer approximation method which render the overall method more efficient. The first concerns an efficient way to evaluate both the regression loss function $c$ and its subgradient $\nabla c$ efficiently. The second discusses a heuristic to compute a warm start $s_1$ to ensure that the first cutting plane added is of high quality, causing the outer approximation algorithm to converge more quickly.

%%%%%%%%%%%%%%%%%
%% SSEC. Efficient dynamic constraint generation
%%%%%%%%%%%%%%%%%
\subsection{Efficient dynamic constraint generation}

In the outer approximation method considered in this document to solve the \ac{cio} problem \eqref{eq:opt:miop:kernel} linear constraints of the type 
\begin{equation}
\label{eq:linearization}
	\eta \geq c(\bar{s}) + \nabla c(\bar{s}) (s - \bar{s})
\end{equation}
at $\bar s$ a given iterate, are considered as cutting planes at every iteration. As such constraints need to be added dynamically, it is essential that we can evaluate both the regression loss function $c$ and its subgradient components efficiently. 

\begin{lemma}[Derivatives of the optimal regression loss $c$]
\label{lemm:derivatives}
Suppose the kernel matrix $K$ is differentiable function of the parameter $s$. Then, we have that the gradient of the regression loss function $c(K) = \frac12 \alpha^\star(K) Y$ can be stated as
\[
	\nabla c(s) =  -\alpha^\star(K) \tpose \cdot \frac{\gamma}{2} \frac{\d K}{\d s} \cdot \alpha^\star(K),
\]
where $\alpha^\star(K)$ maximizes \eqref{eq:regression:dual} and hence is the solution to the linear system $$\alpha^\star(K) = \left(\eye{n} + \gamma K \right)^{-1} Y.$$
\end{lemma}

We note  that the naive numerical evaluation of the convex loss function $c$ or any of its subgradients would require the inversion of the regularized kernel matrix $\eye{n}+\gamma \sum_{j\in[p]} \bar s_j K_j$. The regularized kernel matrix is dense in general and always of full rank. Unfortunately, matrix inversion of general matrices presents work in the order of $\mc O(n^3)$ floating point operations and quickly becomes excessive for sample sizes $n$ in the order of a few 1,000s. Bear in mind that such an inversion needs to take place for each cutting plane added in the outer approximation Algorithm \ref{alg:outer_approximation}. 

It  would thus appear that computation of the regression loss $c$ based on its explicit characterization \eqref{eq:char2} is very demanding. Fortunately, the first explicit characterization \eqref{eq:char1} can be used to bring down the work necessary to $\mc O(k^3+n k)$ floating point operations as we will show now. Comparing equalities \eqref{eq:char1} and \eqref{eq:char2} results immediately in the identity
\begin{equation}
\label{eq:woodbury}
	\alpha^\star(\textstyle\sum_{j\in[p]} s_j K_j) = \left( \eye{n} - X_s (\eye{k}/\gamma+ X_s\tpose X_s)^{-1} X_s \right) Y.
\end{equation}
The same result can also be obtained by applying the matrix inversion lemma stated in \citep{hager1989updating} to the regularized kernel matrix by noting that the micro kernels $K_j$ are rank one dyadic products. The main advantage of the previous formula is the fact that it merely requires the inverse of the much smaller capacitance matrix $C \defn \eye{k}/\gamma+X_s\tpose X_s$ in $\S_{++}^{k}$ instead of the dense full rank regularized kernel matrix in $\S_{++}^{n}$. 

Using expression \eqref{eq:woodbury}, both the regression loss function $c$ and any of its subgradients can be evaluated using $\mc O(k^3+n k)$ instead of $\mc O(n^3)$ floating point operations. When the number of samples $n$ is significantly larger than $k$, the matrix inversion lemma provides a significant edge over a vanilla matrix inversion. We note  that from a statistical perspective this always must be the case if there is any hope that sparse regression might yield statistically meaningful results.

Pseudocode implementing the ideas discussed in this section is provided in Algorithm \ref{alg:regression_function}.

\begin{algorithm}
\SetKwInOut{Input}{input}
\SetKwInOut{Output}{output}
\caption{Regression function and subgradients}
\label{alg:regression_function}
 \Input{$Y \in \Re^{n}$, $X \in \Re^{n\times p}$, $s \in \S^p_k$ and $\gamma \in \Re_{++}$}
 \Output{$c \in \Re_+$ and $\nabla c \in \Re^p$}
 $\alpha^\star \gets Y - X_s (\eye{k}/\gamma+ X_s\tpose X_s)^{-1} X_s\tpose Y$ \\
 $c \gets \frac12 Y\tpose \alpha^\star$ \\
 \For {$j$ in $[p]$ } {
 	$\nabla c_j \gets -\frac{\gamma}{2} (X_j\tpose \alpha^\star)^2$ \\
 }
\end{algorithm}

%%%%%%%%%%%%%%%%%
%% SSEC. Dual warm starts
%%%%%%%%%%%%%%%%%
\subsection{Dual warm starts}

Regardless of the initial selection $s_1$, the outer approximation Algorithm \ref{alg:outer_approximation} will eventually return the optimal subset solution $s^\star$ to the sparse regression formulation in Theorem \ref{thm:cio}. Nevertheless, to improve computational speed in practice it is often desirable to start with a high-quality warm start rather than any arbitrary feasible point in $\S^p_k$. 

As already briefly hinted upon, a high-quality warm start can be obtained by solving the continuous relaxation \eqref{eq:opt:miop:dual}. More specifically, we take as warm start $s_1$ to the outer approximation algorithm the solution to 
\begin{equation}
	s_1 \in \arg \max_{s \in \S_k^p} ~\textstyle\sum_{j\in[p]} s_j \cdot \alpha^{\star\top} K_j \alpha^\star,
	\label{eq:tr1}
\end{equation}
where $\alpha^\star$ is optimal in \eqref{eq:opt:miop:dual}.  Note that the solution to problem \eqref{eq:tr1}   can be found explicitly as the vector indicating the $k$ largest components of $(\alpha^{\star\top} K_1 \alpha^\star, \dots, \alpha^{\star\top} K_p \alpha^\star)$. We finally remark that the \texttt{Lasso} or  the solution found by the first order heuristic developed in \citep{bertsimas2014statistics} could have been used equally well.

%%%%%%%%%%%%%%%%%
%% SEC. Scalability and phase transitions
%%%%%%%%%%%%%%%%%
\section{Scalability and phase transitions}
\label{sec:empirical_results}

To evaluate the effectiveness of the  cutting plane algorithm developed in Section \ref{sec:cutting_plane}, we report its ability to recover the correct regressors as well as its running time. In this section, we present empirical evidence  on  two critically important observations. The first observation is that our cutting plane algorithm scales to provable optimality in seconds for large regression problems with $n$ and $p$ in the 100,000s. That is two orders of magnitude larger than the  known exact sparse regressor methods in  \citep{bertsimas2014statistics} and takes away the main propelling justification for heuristic approaches for many regression instances in practice. The second observation relates to the fact that we observe phase transition phenomena in the three important properties which characterize our exact sparse regression formulation : its ability to find all relevant features ($A\%$), its rejection of irrelevant features from the obfuscating bulk ($F\%$), and the time ($T$) it takes to find an exact sparse regressor using our cutting plane Algorithm \ref{alg:outer_approximation}. 

All algorithms in this document are implemented in \texttt{Julia} and executed on a standard \texttt{Intel(R) Xeon(R) CPU E5-2690 @ 2.90GHz} running \texttt{CentOS release 6.7}. All optimization was done with the help of the commercial mathematical optimization distribution \texttt{Gurobi version 6.5}.

\subsection{Data description}
\label{ssec:data_description}

Before we present the empirical results, we first describe the properties of the synthetic data which shall be used throughout this section. The input and response data are generated synthetically with the observations $Y$ and input data $X$ satisfying the linear relationship
\[
	Y = X w_{\mathrm{true}} + E.
\]
The unobserved true regressor $w_{\mathrm{true}}$ has exactly $k$-nonzero components at indices selected uniformly without replacement from $[f]$. Likewise, the nonzero coefficients in $w_{\mathrm{true}}$ are drawn uniformly at random from the set $\{-1, +1\}$. The observation $Y$ consists of the signal $S \defn X w_{\mathrm{true}}$ corrupted by the noise vector $E$. The noise components $E_i$ for $i$ in $[n]$ are drawn \ac{iid} from a normal distribution $N(0, \sigma^2)$ and scaled to $$\sqrt{\mathrm{SNR}} = \norm{S}_2 / \norm{E}_2$$ Evidently as the \ac{snr} increases, recovery of the unobserved true regressor $w_{\mathrm{true}}$ from the noisy observations can be done with higher precision.

We have yet to specify how the input matrix $X$ is chosen. We assume here that the input data samples $X = (x_1, \dots, x_n)$ are drawn from an \ac{iid} source with Gaussian distribution; that is $$x_i \sim N(0, \Sigma), \quad \forall i\in [n].$$ The variance matrix $\Sigma$ will be parametrized by the correlation coefficient $\rho \in [0, 1)$ as $\Sigma(i, j) \defn \rho^{\abs{i-j}}$ for all $i$ and $j$ in $[p]$. As the $\rho$ tends to $1$, the columns of the data matrix $X$ become more alike which should impede the discovery of nonzero components of the true regressor $w_{\mathrm{true}}$ by obfuscating them with highly correlated look-a-likes. In the extreme case in which $\rho =1$, all columns of $X$ are the same at which point there is no hope of discovering the true regressor $w_{\mathrm{true}}$ even in the noiseless case.

\subsection{Scalability}

We provide strong evidence that the cutting plane Algorithm \ref{alg:outer_approximation} represents a truly scalable algorithm to the exact sparse regression problem \eqref{eq:l0-regression:primal} for $n$ and $p$ in the 100,000s. As many practical regression problems are within reach of our exact cutting plane Algorithm \ref{alg:outer_approximation}, the need for convex surrogate regressors such as \texttt{Elastic Net} and \texttt{Lasso} is greatly diminished. 

We note that an effective  regression must find all relevant features ($A\%=100$) while at the same time reject those that are irrelevant $(F\%=0)$. To separate both efforts, we assume in this and the following  section that   true number $k$ of nonzero components of the ground truth $w_\mathrm{true}$ is known. In this case $A\%+F\%=100$ which allows us to focus entirely on the the accuracy of the obtained regressors. Evidently, in most practical regression instances $k$ needs to be inferred from the data as well. Incorrect determination of this number can indeed lead to high false alarm rates. We will return to this important issue of variable selection and false alarm rates at the end of the subsequent section.

For the sake of comparison, we will also come to discuss the time it takes to solve the \texttt{Lasso} heuristic \eqref{eq:l1-regression:primal} as implemented by the \texttt{GLMNet} implementation of \citet{friedman2013glmnet}.  Contrary to exact sparse regression, no direct way exists to obtain a sparse regressor from solving the convex surrogate heuristic \eqref{eq:l1-regression:primal}. In order to facilitate a fair comparison however, we shall take that \texttt{Lasso} regressor along a path of optimal solutions in \eqref{eq:l1-regression:primal} for varying $\lambda$ which is the least regularized but has exactly $k$ nonzero coefficients as a heuristic sparse solution. 

\begin{table}
  \begin{center}
    %%%%%%%%%%%%%%% 
%%%%%% CIO %%%%%%
%%%%%%%%%%%%%%% 

\begin{tabular}{| c | l | c c c | c c c |}
  \cline{3-8}
  \multicolumn{2}{c}{} 											&	\multicolumn{3}{|c|}{Exact $T$ [s]} 		     		& 	\multicolumn{3}{c|}{\texttt{Lasso} $T$ [s]} \\
  \cline{3-8}
  \multicolumn{2}{c|}{} 							 			 &	$n=10$k	&	$n=20$k	&	$n=100$k	&	$n=10$k	&	$n=20$k	&	$n=100$k	\\
  \cline{1-8}
  \multirow{3}{*}{\rotatebox{90}{$k=10~$}}			&	$p = 50$k 	& 	21.2 		& 	34.4 		& 	310.4	&	69.5 		& 	140.1	& 	431.3 \\
                                                                                                        &	$p = 100$k 	& 	33.4 		& 	66.0 		& 	528.7	& 	146.0	& 	322.7 	& 	884.5 \\
                                                                                                        &	$p = 200$k 	& 	61.5 		& 	114.9 	& 	NA		& 	 279.7 	& 	566.9 	& 	NA \\
  \cline{1-8}
  \multirow{3}{*}{\rotatebox{90}{$k=20~$}}			&	$p = 50$k 	& 	15.6 		&	38.3 		& 	311.7	&	107.1 	& 	142.2 	& 	467.5 \\
                                                                                                        &	$p = 100$k 	& 	29.2 		& 	62.7 		& 	525.0	& 	216.7 	& 	332.5 	& 	988.0 \\
                                                                                                        &	$p = 200$k 	& 	55.3 		& 	130.6 	& 	NA		& 	353.3 	& 	649.8 	& 	NA \\
  \cline{1-8}
  \multirow{3}{*}{\rotatebox{90}{$k=30~$}}			&	$p = 50$k 	& 	31.4	 	& 	52.0 		& 	306.4	&	99.4 		& 	220.2 	& 	475.5 \\
                                                                                                        &	$p = 100$k 	& 	49.7 		& 	101.0 	& 	491.2	& 	318.4 	& 	420.9 	& 	911.1 \\
                                                                                                        &	$p = 200$k 	& 	81.4 		& 	185.2 	& 	NA		&	480.3 	& 	884.0 	&	NA \\
  \cline{1-8}
\end{tabular}
  \end{center}
\vspace{0.5em}
\caption{A comparison between exact sparse regression using our cutting plane algorithm and the \texttt{Lasso} heuristic with respect to their solution time in seconds applied to noisy ($\sqrt{\mathrm{SNR}}=20$) and lightly correlated data ($\rho=0.1$) explained by either $k=10$, $k=20$ or $k=30$ relevant features. These problem instances are truly large scale as for the largest instance counting $n=100,000$ samples for $p=200,000$ regressors a memory exception was thrown when building the data matrices $Y$ and $X$. Remarkably, even on this scale the cutting plane algorithm can be significantly faster than the \texttt{Lasso} heuristic.
}
\label{table:scalability}
\end{table}

In Table \ref{table:scalability} we discuss the timing results for exact sparse linear regression as well as for the \texttt{Lasso} heuristic applied to noisy ($\sqrt{\mathrm{SNR}}=20$) and lightly correlated ($\rho=0.1$) synthetic data. We do not report the accuracy nor the false alarm rate of the obtained solution as this specific data is in the regime where exact discovery of the support occurs for both the \texttt{Lasso} heuristic and exact sparse regression. %We discuss phase transition phenomena in the subsequent section.

Remarkably, the timing results in Table \ref{table:scalability} suggest that using an exact method does not impede our ability to obtain the solution fast. The problem instances displayed are truly large scale as indeed for the largest problem instance a memory exception was thrown when building the data matrices $X$ and $Y$. In fact, even in this large scale setting our cutting plane algorithm can be significantly faster than the \texttt{Lasso} heuristic. Admittedly though, the \texttt{GLMNet} implementation returns an entire solution path for varying $\lambda$ instead of a single regression model. Comparing though to the performance reported on exact sparse regression approaches in \citep{furnival2000regressions} and \citep{bertsimas2014statistics}, our method presents a potentially game changing speed up of at least two orders of magnitude. The results in Table \ref{table:scalability} thus do refute the widely held belief that exact sparse regression is not feasible at large scales. In fact, we consider pointing out the fact that exact sparse regression is not hopeless in practice an 
important contribution of this paper.

Although a hard theoretical picture is not yet available as for why the cutting plane Algorithm \ref{alg:outer_approximation} proves so efficient, we hope that these encouraging results spur an interest in exact approaches towards sparse regression. In the subsequent section, we will come to see that the scalability of exact sparse regression entails more than meets the eye.

\subsection{Phase transition phenomena}
\label{ssec:phase_transition_phenomena}

We have established that the cutting plane Algorithm \ref{alg:outer_approximation} scales to provable optimality for problems with number of samples and regressor dimension in the 100,000s. Let us remark that for the results  presented in Table \ref{table:scalability}, both the exact and heuristic algorithms returned a sparse regressor with correct support and otherwise were of similar precision. In cases where the data does not allow a statistically meaningful recovery of the ground truth $w_{\mathrm{true}}$ an interesting phenomenon occurs. We present and discuss in this part of the paper three remarkable phase transition phenomena. The first will concern the statistical power of sparse regression, whereas the second will concern our ability to find the optimal sparse regressor efficiently. We will refer to the former transition as the accuracy transition, while referring to the latter as the complexity transition. The false alarm phase transition is the third phase transition phenomenon and relates to the ability of exact sparse regression to reject irrelevant features from the obfuscating bulk. We will argue here using strong empirical evidence that these transitions are in fact intimately related. Of all three phase transitions discussed here, only the accuracy phase transition has previously received attention and is also understood theoretically.

The accuracy phase transition describes the ability of the sparse regression formulation \eqref{eq:l0-regression:primal} to uncover the ground truth $w_{\mathrm{true}}$ from corrupted measurements alone. The corresponding phase transition for the \texttt{Lasso} has been extensively studied in the literature by amongst many others \citet{buhlmann2011statistics, hastie2015statistical} and \citet{wainwright2009sharp} and is considered well understood by now. As mentioned, with uncorrelated input data ($\rho=0$) a phase transition occurs at the curve \eqref{eq:wainwright}. In the regime $n > n_1$ exact recovery with \texttt{Lasso} occurs with high-probability for some $\lambda>0$, whereas otherwise the probability for successful recovery drops to zero.

A similar phase transition has been observed by \citet{zheng2015does} and \citet{wang2011performance} for exact sparse regression as well, although this transition is far less understood from a theoretical perspective than the similar transition for its heuristic counterpart. Recently though, \citet{gamarnik2017high} have made some way and shown that an all or nothing phase transition phenomena occurs for exact sparse regression with binary coefficients as well.

\begin{theorem}[\citet{gamarnik2017high}]
  \label{thm:gamarnik}
  Let the data ($\rho=0$) be generated as in Section \ref{ssec:data_description}. Let $\epsilon>0$. Suppose $k \log k \leq C n$, for some $C> 0$ for all $k$ and $n$. Suppose furthermore that $k\to\infty$ and $\sigma^2/k\to 0$. If $n\geq (1-\epsilon) n_0$, then with high probability
  \[
    \frac1k \norm{w_0^\star-w_{\mathrm{true}}}_0 \to 0.
  \]
  Whereas when $n \leq (1-\epsilon) n_0$, then with high probability
  \(
    \frac1k \norm{w_0^\star-w_{\mathrm{true}}}_0 \to 1.
  \)
\end{theorem}

Although the following theorem holds for unregularized sparse regression ($\gamma \to \infty$), the same holds for other appropriately chosen values of the regularization parameter as well. Interestingly, \citet{gamarnik2017high} the proof technique of Theorem \ref{thm:gamarnik} might give additional intuitive insight with regard to the phase transition phenomena with respect to the statistical accuracy and computational complexity of exact sparse regression problem, which we will now empirically report on.

\begin{figure}
\begin{center}
    \includegraphics[width=1\textwidth]{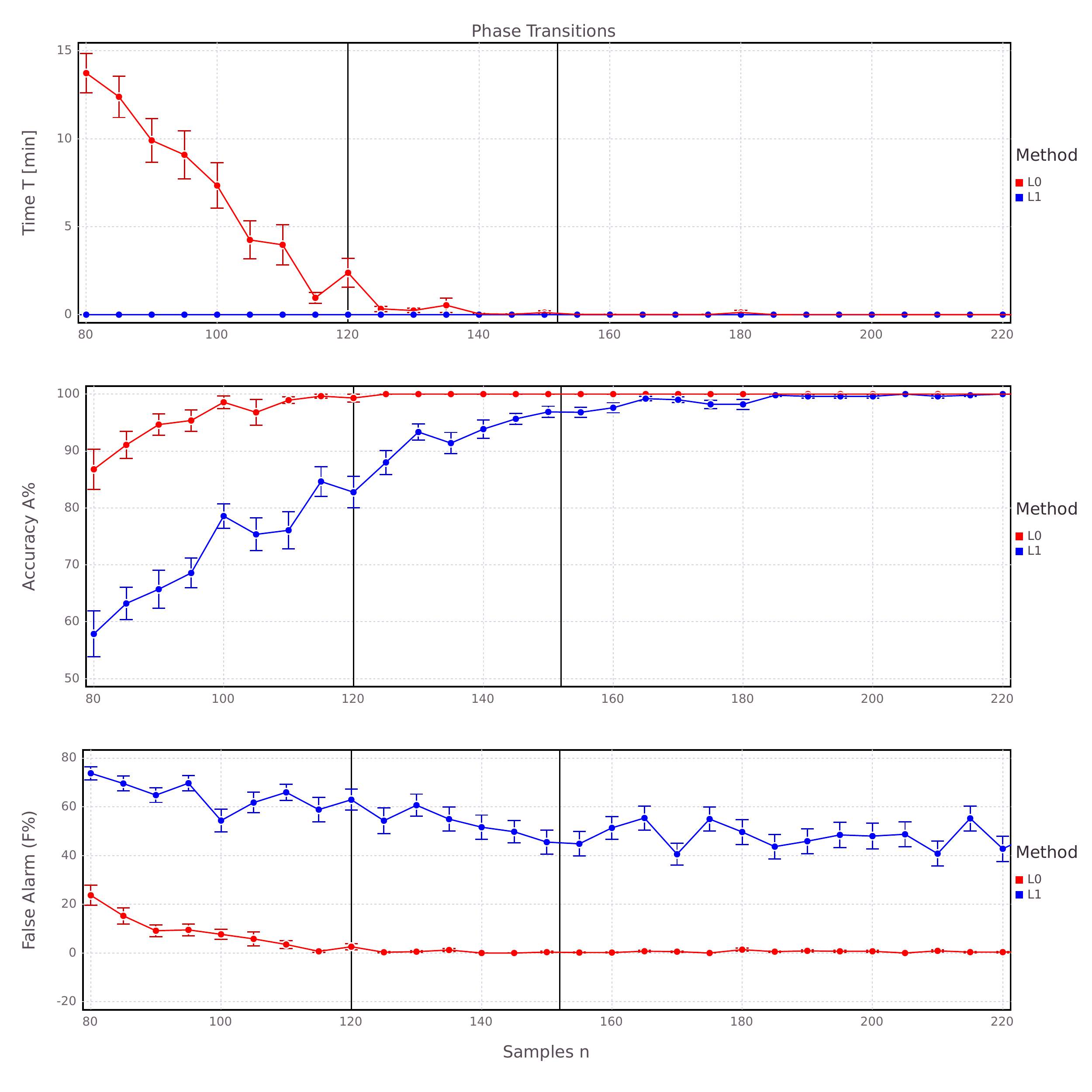}
\end{center}
\caption{A comparison between exact sparse regression using our cutting plane algorithm and the approximate \texttt{Lasso} heuristic on uncorrelated data ($\rho=0$) with noise ($\sqrt{SNR}=20$) counting $p=2,000$ regressors of which only $k=10$ are relevant. In the top panel we depict the time in minutes necessary to solve the sparse regression problem using either method as a function of the number of samples. The panel below gives the corresponding accuracy $A\%$ of the regressors as a function of the number of samples. The red vertical line at $n_1=152$ samples depicts the accuracy phase transition concerning the ability of the \texttt{Lasso} heuristic to recover the support of the ground truth $w_{\mathrm{true}}$. The blue vertical line at $n_t=120$ does the same for exact sparse regression. The final panel indicates the ability of both methods to reject obfuscating features in terms of the false alarm rate $F\%$. It can thus be seen that exact sparse regression does yields more statistically meaningful regressors (higher accuracy $A\%$ for less false alarms $F\%$) than the \texttt{Lasso} heuristic. Furthermore, a complexity phase transition can be recognized as well all around $n_t$.
}
\label{fig:cio_l1}
\end{figure}

In Figure \ref{fig:cio_l1}, we show empirical results for noiseless uncorrelated synthetically generated data with  $p=2,000$   of which only $k=10$ are relevant. The accuracy $A\%$ and false alarm rates $F\%$  using exact sparse regression as well as the \texttt{Lasso}  and time $T$ in minutes  to obtain either one are taken as the average values of fifty independent synthetic datasets. When the optimal solution is not found in less than fifteen minutes we take the best solution found up to that point. The error bars give an indication of one inter-sample standard deviation among these fifty independent experiments. The colored horizontal lines indicate that the number of samples $n$ after which either method returned   a full recovery ($A\%=100$) of the support of the ground truth when both are given the correct number $k$ of relevant sparse features. The \texttt{Lasso} heuristic is empirically found to require approximately $n=180$ samples to recover the true support which corresponds rather well with the theoretically predicted $n_1=152$ necessary samples by \citet{wainwright2009sharp}. Unsurprisingly, the related accuracy phase transition of exact sparse regression using Algorithm \ref{alg:outer_approximation} is found empirically to occur  at $n_t=120$ samples.

We now discuss the second transition which indicates that the time it takes to solve the sparse regression \eqref{eq:l0-regression:primal} using the cutting plane Algorithm \ref{alg:outer_approximation} experiences a phase transition as well. We seem to be the first to have seen this complexity phase transition likely due to the fact that scalable algorithms for exact sparse regression have historically been lacking. Nevertheless, the fact that the complexity of exact sparse regression might experience a phase transition has been allude to before. Contrary to traditional complexity theory which suggests that the difficulty of a problem increases with problem size, the sparse  regression problem  has the property that as the number of samples  $n > n_t$ increases the problem becomes easier in that the solution recovers 100\% of the true signal, and our approach solves the problem extremely fast (in fact faster than \texttt{Lasso}), while for small number of samples $n < n_t$ exact sparse regression seems impractical. It should be remarked that as $n_0 \approx 50 < n_t$ there still remains a region in which exact sparse regression is statistically relevant but computationally not feasible.

In all the experiments conducted up to this point, we assumed that the number of non-zero regressor coefficients $k$ of the ground truth $w_\mathrm{true}$ underlying the data was given. Evidently, in most practical applications the sparsity parameter $k$ needs to be inferred from the data as well. In essence thus, any practical sparse regression procedure must pick those regressors contributing to the response out of the obfuscating bulk. To that end, we introduced the false alarm rate $F\%$ of a certain solution $w^\star$ as the percentage of regressors selected which are in fact unfitting. The ideal method would of course find all contributing regressors ($A\%=100$) and not select any further ones ($F\%=0$). In practice clearly, a trade-off must sometimes be made. The final phase transition will deal with the ability of exact sparse regression to reject obfuscating irrelevant features using cross validation.

Historically, cross validation has been empirically found to be an effective way to infer the sparsity parameter $k$ from data. Hence, for both exact sparse regression and the \texttt{Lasso} heuristic, we select that number of non-zero coefficients which generalizes best to the validation sets constructed using cross validation with regards to prediction performance. In case of exact sparse regression, we let $k$ range between one and twenty whereas the true unknown number of non-zero regressors was in fact ten. The third plot in Figure \ref{fig:cio_l1} gives the false alarm rate $F\%$ of both methods in terms of the number of samples $n$. As can be seen, the \texttt{Lasso} heuristic has difficulty keeping a low false alarm rate with noisy data. Even in the region where the \texttt{Lasso} heuristic is accurate ($A\%$), it is not as sparse as hoped for. Exact sparse regression does indeed yield sparser models as it avoids including regressors that do not contribute to the observations.

\subsection{Parametric Dependency}
\label{ssec:param-depend}
To investigate the effect of each of the data parameters even further, we use synthetic data with the properties presented in Table \ref{table:parameters}. In order to be able to separate the effect of each parameter individually, we present the accuracy $A \%$, false alarm rate $F \%$ and solution time $T$ of our cutting plane algorithm as a function of the number of samples $n$ for each parameter value separately while keeping all other parameters fixed to their nominal value. All results are obtained as the average values of twenty independent experiments. The figures in the remainder of this section indicate that the accuracy, false alarm and complexity phase transitions shown in Figure \ref{fig:cio_l1} persist for a wide variety of properties of the synthetic data.

\begin{table}
\begin{center}
\begin{tabular}{llr}
 \hline
 Sparsity & $k$ & $\{10^\star, 15, 20\}$ \\
 Dimension & $p$ & $\{5000^\star, 10000, 15000\}$ \\
 Signal-to-noise ratio & $\sqrt{\mathrm{SNR}}$ & $\{3, 7, 20^\star\}$ \\
 \hline
\end{tabular}
\end{center}
\caption{Parameters describing the synthetic data used in Section \ref{ssec:param-depend}. The starred values denote the nominal values of each parameter.}
\label{table:parameters}
\end{table}

\subsubsection*{Feature dimension $p$}

As both phase transition curves \eqref{eq:wainwright} and \eqref{eq:gamarnik} depends only logarithmically on $p$, we do not expect the reported phase transitions to be very sensitive to the regressor dimension either. Indeed, in Figure \ref{fig:cio_p} only a minor influence on the point of transition between statistically meaningful and efficient sparse regression to unreliable and intractable regressors is observed as a function of  $p$. 

\begin{figure}
\begin{center}
    \includegraphics[width=0.85\textwidth]{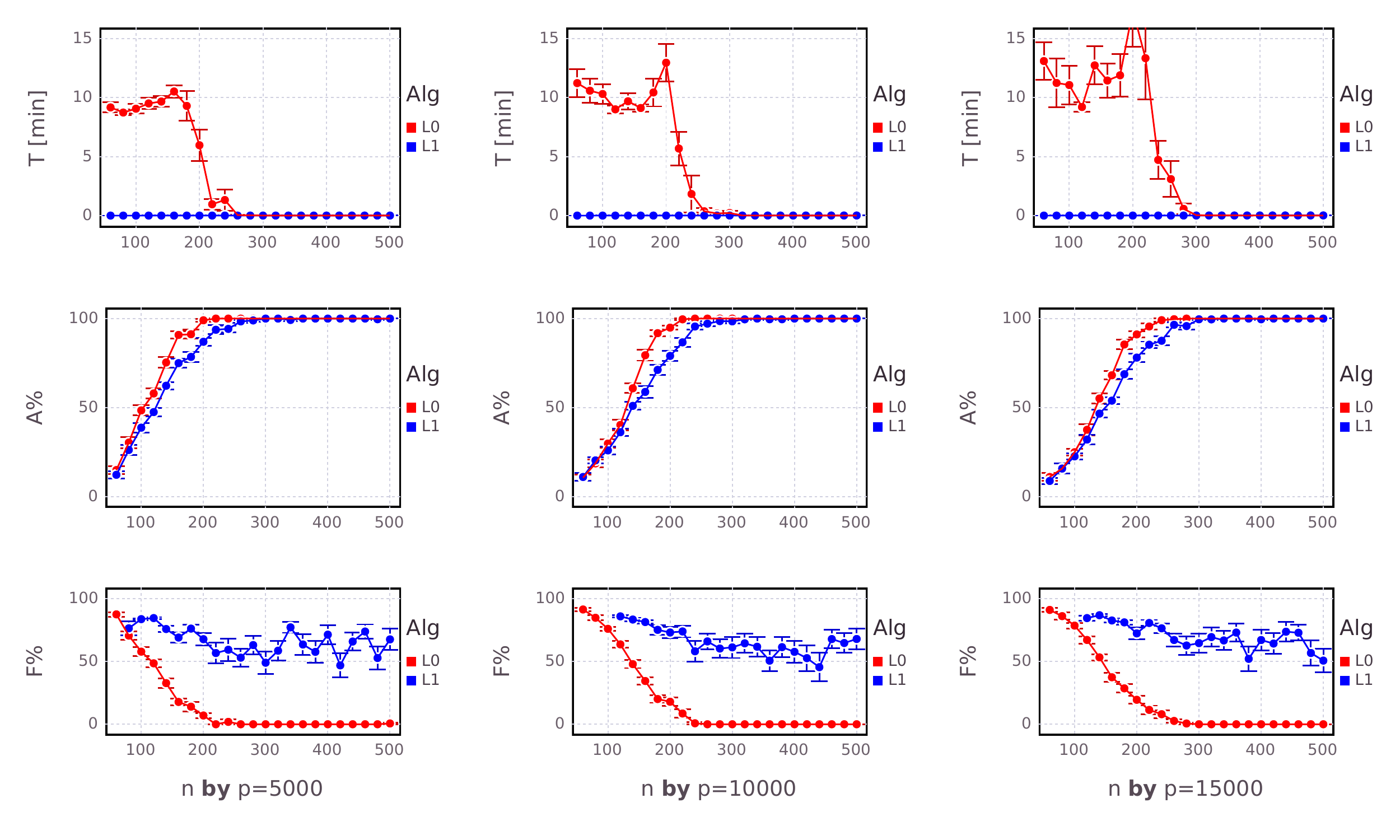}
\end{center}
\vspace{-10pt} 
\caption{The   top  panel  shows the time it takes to solve the sparse regression problem using the cutting plane method for data with $p = 5,000$, $10,000$ or $15,000$ regressors 
%and the nominal properties given in Table \ref{table:parameters}
 as a function of $n$. When the optimal solution is not found in less than ten minutes we take the best solution found up to that point. The bottom panels show the accuracy $A\%$ and false alarm rate $F\%$. Only a minor influence on the point of transition between statistically meaningful and efficient sparse regression to unreliable and intractable regression is observed as a function of the regression dimension $p$.
}
\label{fig:cio_p}
\end{figure}

\subsubsection*{Sparsity level $k$}

Figure \ref{fig:cio_k} suggests that $k$  has an important influence of the phase transition curve. The experiments suggest that there is a  threshold $f_t$ such that if $n/k \geq f_t$, then full support recovery $(A\%=100, \,F\%=0)$ occurs and the time to solve problem \eqref{eq:l0-regression:primal} is in the order of seconds and only grows linear in $n$. Furthermore, if $n/k < f_t$, then support recovery $A\%$ drops to zero, false alarms $F\%$ surge, while the time to solve problem \eqref{eq:l0-regression:primal} grows combinatorially as $\binom{p}{k}$.  This observation is in line with the theoretical result \eqref{eq:gamarnik}, which predicts that this threshold only depends logarithmically on the feature dimension $p$ and the \ac{snr} which we study subsequently.

\begin{figure}
\begin{center}
    \includegraphics[width=0.85\textwidth]{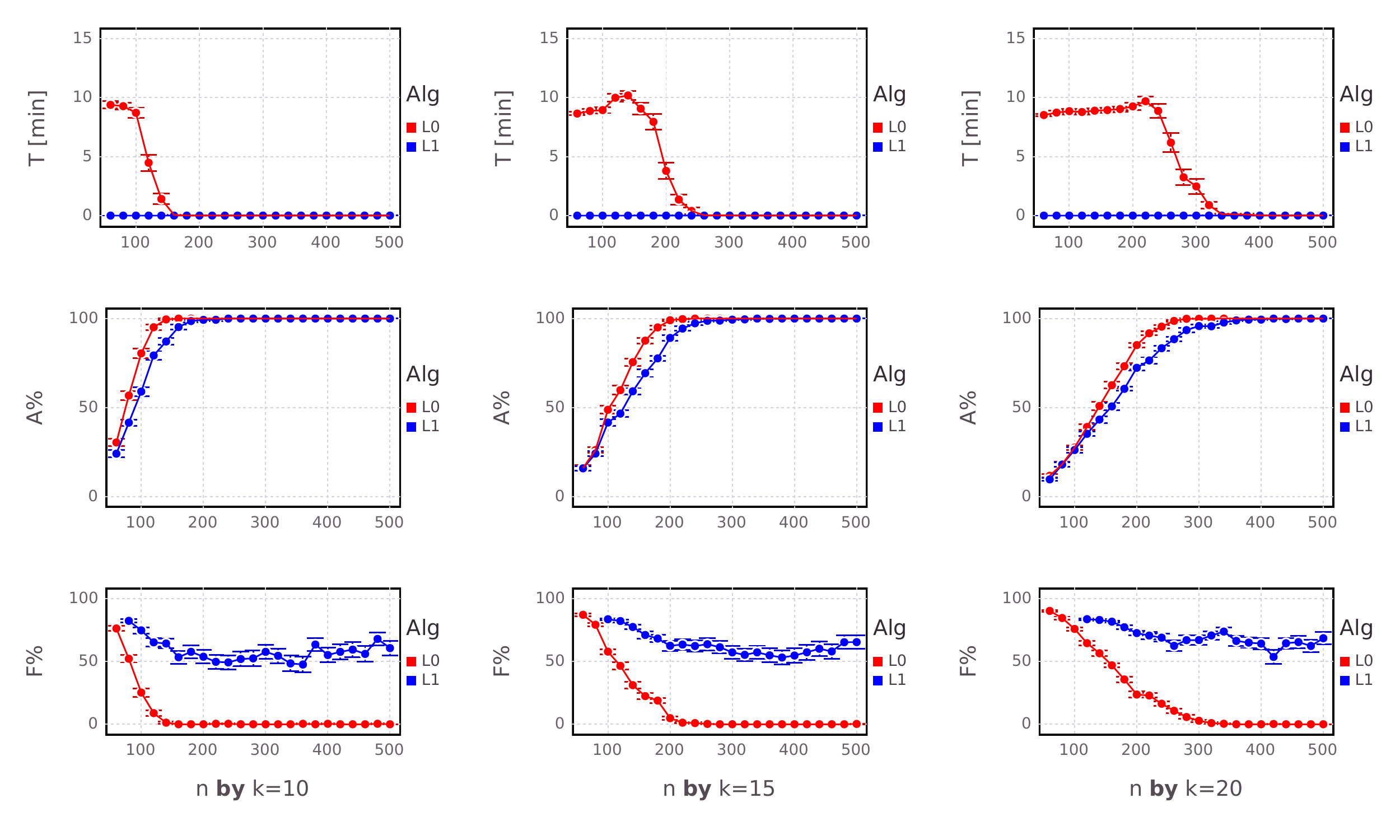}
\end{center}
\vspace{-10pt} 
\caption{The   top  panel shows the time it takes  to solve the sparse regression problem as a function of $n$
 using the cutting plane method for data with $p=5,000$ regressors of which only $k = 10$, $15$ or $k=20$ are relevant. When the optimal solution is not found in less than ten minutes we take the best solution found up to that point. The bottom panels show the accuracy $A\%$ and false alarm rate $F\%$. These results suggest that the quantity $n/k$ is a major factor in the phase transition curve of exact sparse regression.}
\label{fig:cio_k}
\end{figure}

\subsubsection*{Signal-to-noise ratio ($\mathrm{SNR}$)}

From an information theoretic point of view, the \ac{snr} must play an important role as well as reflected by the theoretical curve \eqref{eq:gamarnik}. Indeed, the statistical power of any method is  questionable when the noise exceeds the signal in the data. In Figure \ref{fig:cio_snr} this effect of noise is observed as for noisy data the phase transition occurs later than for more accurate data.

\begin{figure}
\begin{center}
    \includegraphics[width=0.85\textwidth]{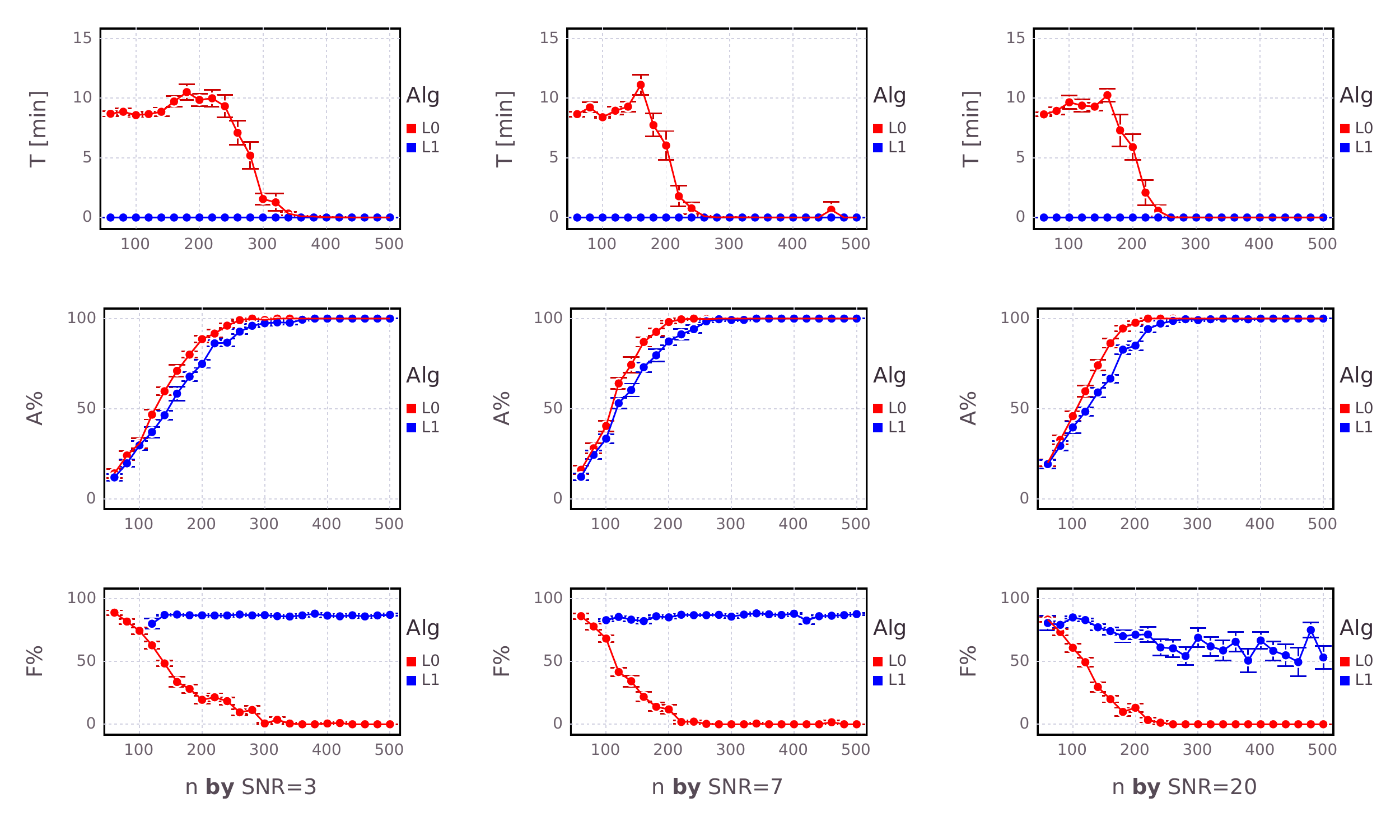}
\end{center}
\vspace{-10pt} 
\caption{The   top  panel  shows the time it takes to solve the sparse regression problem as a function of $n$
using the cutting plane method for data with signal-to-noise level $\sqrt{\mathrm{SNR}} = 3,$ $7$  and $20$.
%and further nominal properties as given in Table \ref{table:parameters}.
When the optimal solution is not found in less than one minute we take the best solution found up to that point. The bottom panel   shows the accuracy $A\%$. }
%As can be expected from an information theoretic point of view, meaningful regression is harder when the data is noisy.}
\label{fig:cio_snr}
\end{figure}

\subsection{A remark on complexity}

The empirical results in this paper suggest that the traditional complexity point of view might be  misleading towards a better understanding of the complexity of the sparse regression problem \eqref{eq:l0-regression:primal}. Indeed, contrary to traditional complexity theory which suggests that the difficulty of a problem increases with dimension, the sparse  regression problem \eqref{eq:l0-regression:primal} has the property that for small number of samples $n$, our approach takes a  large amount of time to solve the problem. However, for a large number of samples $n$, our approach solves the problem extremely fast and recovers 100\% of the support of the true regressor $w_{\mathrm{tru
e}}$. 
% It  to the authors that branding the exact sparse regression problem ``hard'' does not do justice to the results presented in this work.

%%%%%%%%%%%%%%%%%
%% SEC. On the road towards nonlinear feature discovery
%%%%%%%%%%%%%%%%%
 \section{The road towards nonlinear feature discovery}
 \label{sec:nonlinear}

In this section, we discuss an extension of the sparse linear regression to the case of nonlinear regression
by augmenting the input data $X$ with  auxiliary nonlinear transformations.  In fact, the idea of nonlinear regression as linear regression to lifted data  underpins kernel methods. Kernel methods can in a primal perspective be viewed as Tikhonov regularization between the observations $Y$ and transformed versions $\psi(x_i)$ of the original data samples. The feature map $\psi(\cdot)$ encodes which nonlinearities should be detected. 

To illustrate the idea  
we augment each of the $p$ original regressors with the  following nonlinear transformations: 
 \begin{equation}
 \label{eq:temp12}
  x, ~\sqrt{\abs{x}},~\log \abs{x}, ~x^2,~x^3,~\cos(10 \pi x),~ \sin(x),~\tanh(2 x).
  \end{equation}
The method could be made   more general by allowing for nonlinear products between variables but we abstain from doing so for the sake of simplicity. To enforce a sparse regression model, we demand that the final regressor can only depend on $k$ different (potentially nonlinear) features. 

 Instead of solving   problem \eqref{eq:l0-regression:primal}, we then solve its nonlinear version 
 \begin{equation}
 \label{eq:l0-regression:nonlinear}
 \begin{array}{rl}
 	\min 	& \frac{1}{2\gamma} \norm{\tilde w}^2_2 + \frac{1}{2} \norm{Y - \psi(X) \tilde w}_2^2 \\[0.5em]
 	\st    	& \norm{\tilde w}_0 \leq k,
 \end{array}
 \end{equation}
 where the matrix $\psi(X)$ in $\Re^{n \times f}$ consists of the application of the transformations in \eqref{eq:temp12} to the input matrix $X$. The nonlinear sparse regression problem \eqref{eq:l0-regression:nonlinear} can be dealt with in an identical manner as its linear counterpart \eqref{eq:l0-regression:primal}. Notice that the dimension of the nonlinear regressor $\tilde w$ is potentially much larger than its linear counterpart $w$. 
 \begin{corollary}[Sparse nonlinear regression]
 \label{cor:cio_nonlinear}
 The sparse regression problem \eqref{eq:l0-regression:nonlinear} can be reformulated as the nonlinear optimization problem 
 \begin{equation*}
 \begin{array}{rl}
 	\min_{s\in \S^f_k} 	& \displaystyle \frac12 Y\tpose \left(\eye{n} + \gamma \textstyle\sum_{j\in[f]} s_j K_j \right)^{-1} Y
 \end{array}
 \end{equation*}
 where  $ K_j \defn \psi_j(X) \psi_j(X) \tpose.$
 \end{corollary}
Note that he only material difference between  Corollary \ref{cor:cio_nonlinear} and Theorem \ref{thm:cio} is the definition of  kernel matrices $K_j$.

 As an illustration of the nonlinear approach described above, consider observations and data coming from the following nonlinear model
 \begin{equation}
 \label{eq:nonlinear_model}
 	Y = 3 \sqrt{\abs{X_4}} -2 X_2^2 + 4 \tanh(2 X_3) + 3 \cos(2 \pi X_2) -2 X_1 + a X_1 X_2 + E.
 \end{equation}
 We assume  that the input data $X$ and noise $E$ is generated using the method outlined in Section \ref{ssec:data_description}. That is, the signal-to-noise ratio was chosen to be $\sqrt {\mathrm{SNR}}=20$ to simulate the effect of noisy data. For simplicity we assume the original data $X$ to be uncorrelated ($\rho =0$). 
   An additional 16 regressors are added to obfuscate the four relevant regressors in the nonlinear model \eqref{eq:nonlinear_model}. The input data after the nonlinear transformations
   in \eqref{eq:temp12} comprised a total of $f=160$ nonlinear features. We consider two distinct nonlinear models for corresponding parameter values $a=0$ and $a=1$. Notice that for the biased case $a=1$, the term $a X_1 X_2$ will prevent our nonlinear regression approach to find the true underlying nonlinear model \eqref{eq:nonlinear_model} exactly.
 
 We state the results of our nonlinear regression approach applied to the nonlinear model \eqref{eq:nonlinear_model} for both $a=0$ and $a=1$ in Table \ref{table:nonlinear}. All reported results are the median values of five independent experiments. Cross validation on $k$ ranging between one and ten was used to determine the number of regressors considered. Determining the best regressor for each $k$ took around ten seconds, thus making a complete regression possible in a little under two minutes.  As currently outlined though, our nonlinear regression approach is not sensitive to nonlinearities appearing as feature products and consequently it will treat the term $a X_1 X_2$ as noise. Hence, the number of underlying regressors we can ever hope to discover is five. For $a=0$,   200 samples suffice to identify the correct nonlinearities and features. For $a=1$ 
  Table \ref{table:nonlinear}  reports an increased false alarm rate compared to $a=0$.

 \begin{table}
   \begin{center}
     \begin{tabular}{| l | l | c c c c c |}
       \cline{2-7}
       \multicolumn{1}{c|}{} & Quality $w^\star$ & $n = 100$ & $n = 200$ & $n = 300$ & $n = 400$ & $n = 500$\\
       \hline 
       $a = 0$ & $(A\%, F\%)$ & (100, 38) & (100, 0) & (100, 0) & (100, 0) & (100, 0)\\
       \hline
       $a = 1$ & $(A\%, F\%)$ & (80, 50) & (100, 17) & (100, 17) & (100, 28) & (100, 17) \\
       \hline
     \end{tabular}
   \end{center}
   \vspace{0.5em}
   \caption{For the nonlinear model \eqref{eq:nonlinear_model} and for $a=0$,   $n=200$  suffice to identify the correct features. 
     For  $a=1$,    $A\%=100$ for $n\geq 200$, but    $F\%>0$.}
   \label{table:nonlinear}
 \end{table}

The method proposed here serves only as an illustration. 
Off course no method can aspire to discover arbitrary nonlinearities without sacrificing its statistical power. 
We believe that this constitutes a promising new road towards nonlinear feature discovery in data. With additional research, we believe that it can become a fierce and more disciplined competitor towards the more ``black box'' approaches such as neural networks.

%%%%%%%%%%%%%%%%%
%% SEC. Conclusions
%%%%%%%%%%%%%%%%%
\section{Conclusions}
\label{sec:conclusions}

We  presented a novel binary convex reformulation  and a novel cutting plane algorithm that solves to provable optimality exact sparse regression problems for instances with sample sizes and regressor dimensions well in the 100,000s. This  presents an improvement of two orders of magnitude compared to known exact sparse regression approaches and takes away the computational edge attributed to sparse regression heuristics such as the \texttt{Lasso} or \texttt{Elastic Net}.

The ability to solve sparse regression problems for very high dimensions allows us to observe new phase transition phenomena. Contrary to   complexity theory which suggests that the difficulty of a problem increases with problem size, the sparse regression problem has  the property that as    $n$ increases, the problem becomes easier in that the solution perfectly recovers the support of the true signal, and our approach solves the problem extremely fast (in fact faster than \texttt{Lasso}), whereas for small   $n$, our approach takes a  large amount of time to solve the problem. 
We  further provide preliminary evidence  that our  methods  open a new road towards nonlinear feature discovery based on sparse selection from a potentially huge amount of desired nonlinearities.

\section*{Acknowledgements}

The second author is generously supported by the Early Post.Mobility fellowship No.\ 165226 of the Swiss National Science Foundation.

\setlength\bibitemsep{10pt}
\printbibliography

\end{document}